\documentclass[11pt]{amsart}
\usepackage{amssymb}
\usepackage{amscd}
\usepackage{amsmath}
\usepackage{comment}
\usepackage{color}
\usepackage[dvipsnames]{xcolor}
\usepackage{url}
\usepackage{hyperref}

\newtheorem{theorem}{Theorem}[section]
\newtheorem{corollary}[theorem]{Corollary}
\newtheorem{lemma}[theorem]{Lemma}
\newtheorem{proposition}[theorem]{Proposition}

\newtheorem*{theorem*}{Theorem (Foia\c{s}, Hamid, Onica and Pearcy, \cite{foias})}

\newtheorem*{theorem**}{Theorem}

\newenvironment{demode}
  {\noindent {{\it Proof of }}}%
  {\par \hfill \fbox{}}

\theoremstyle{definition}

\theoremstyle{definition}

\newtheorem{remark}[theorem]{Remark}

\newcommand{\EL}{\mathcal{L}}

\newcommand{\C}{\mathbb{C}}
\newcommand{\N}{\mathbb{N}}

\newcommand{\ran}{\textnormal{ran}}

\newcommand{\HH}{\mathcal{H}}

\newcommand{\Span}{ {\textnormal{span}}  }

\newcommand{\Id}{\operatorname{\,I}}

\begin{document}
	\title[Hyperinvariant subspaces of hyponormal operators]{Hyperinvariant subspaces of hyponormal operators: A constructive decomposition approach}

    \author{Norberto Clemente and Eva A. Gallardo-Gutiérrez}
	\address{Norberto Clemente  and Eva A. Gallardo-Gutiérrez  \newline
		Departamento de Análisis Matemático y Matemática Aplicada,\newline
		Facultad de CC. Matematicas,
		\newline Universidad Complutense de
		Madrid, \newline
		Plaza de Ciencias N$^{\underbar{\Tiny o}}$ 3, 28040 Madrid,  Spain
		\newline
		and Instituto de Ciencias Matemáticas ICMAT (CSIC-UAM-UC3M-UCM),
		\newline Madrid,  Spain }
\email{norberto.clemente@icmat.es}
	
\email{eva.gallardo@mat.ucm.es}


\thanks{Both authors are partially supported by Plan Nacional  I+D grant no.  PID2022-137294NB-I00, Spain,
		the Spanish Ministry of Science and Innovation, through the ``Severo Ochoa Programme for Centres of Excellence in R\&D'' (CEX2019-000904-S \&
CEX2023-001347) and from the Spanish
National Research Council, through the ``Ayuda extraordinaria a Centros de Excelencia Severo Ochoa'' (20205CEX001). \newline
Second author also acknowledges support of the  Grant PRE2022-101849 funded by: MICIU/AEI/ 10.13039/501100011033 and FSE+.}



\date{November, 2025}

\begin{abstract}
It is shown that any hyponormal operator on an infinite-dimensional separable Hilbert space that admits a decomposition \( T = R + V \), where \( R \) is tridiagonal and \( V \) is trace-class, has nontrivial closed hyperinvariant subspaces provided $T$ is not a multiple of the identity. We further discuss implications of this result for the invariant subspace problem of hyponormal operators answering, in particular,  negatively to a question raised by Kim and Lee \cite{kimlee} regarding an explicit approach to such a problem.  Finally, we characterize the existence of reducing subspaces for hyponormal operators addressing an approach by Aronszajn and Smith.
\end{abstract}

\maketitle

\section{Introduction}

Among bounded linear operators on an infinite dimensional complex separable Hilbert space \( \HH \), normal operators are, perhaps, the most thoroughly understood class because of their structural simplicity and rich spectral theory.  Nevertheless, despite such a fact, one of the most challenging problems in Operator Theory that remains nowadays is the question of whether a compact perturbation of a normal (or of a diagonalizable normal or of, even, a self-adjoint) operator necessarily has nontrivial closed invariant subspaces.

\smallskip

This problem has a long and distinguished history: Livšić~\cite{Livsic1957} first solved it for nuclear perturbations, Sahnovič~\cite{Sahnovic1962} extended the result to Hilbert--Schmidt perturbations, and Gohberg and Krein~\cite{GohbergKrein1957}, Macaev~\cite{Macaev1961}, and Schwartz~\cite{Schwartz1963} further generalized it to perturbations in the Schatten--von Neumann class \( \mathcal{C}_p \) for \( 1 \le p < \infty \) (see also~\cite{Kato1995} for further references). 

\smallskip

When restricting normal operators to one of its invariant subspaces (namely, when dealing with \emph{subnormal operators}), a fundamental milestone in the study of invariant subspaces was proved by Brown \cite{scottbrown1}, showing that subnormal operators (or, obviously, those unitarily equivalent to them) have nontrivial closed invariant subspaces.

\smallskip

Moreover, Brown \cite{scottbrown2} proved that every hyponormal operator with \emph{thick spectrum} has nontrivial closed invariant subspaces. Recall that a linear bounded operator $T$ in $\HH$ is \emph{hyponormal} if its self-commutator $[T^{*},T]=T^{*}T-TT^{*}$ is semi-definite positive and every subnormal operator appears to be hyponormal. Nevertheless, it remains unsolved  whether every hyponormal operator admits a nontrivial closed invariant subspace and, indeed, a theorem of Berger and Shaw \cite{bergershaw} shows that the invariant subspace problem for hyponormal operators can be reduced to a special instance of the aforementioned problem regarding compact perturbations of normal operators (see ~\cite[Corollary~8.5]{Pearcy}):

\medskip
\noindent
\begin{quotation}
If every operator \( T = N + K \), where \( N \) is normal and \( K \) is compact, has a nontrivial invariant subspace, then every hyponormal operator has a nontrivial closed invariant subspace.
\end{quotation}
\medskip

At this regards, it is worthy to remark that, recently, the authors of \cite{GG4} proved that a large subclass of trace-class perturbations of normal operators  are decomposable operators in the sense of  Colojoar\u{a} and Foia\c{s} \cite{CF68}, extending previous theorems of Foia\c{s}, Jung, Ko and Pearcy \cite{FJKP07,FJKP08,FJKP11}, Fang and J. Xia \cite{FX12}, and those in \cite{GG,GG3} on an open question posed by Pearcy in the seventies ~\cite[Problem~K]{Pearcy}. Moreover, every operator $T$ in such a class that is not a multiple of the identity operator has a rich spectral structure and plenty of nontrivial closed \emph{hyperinvariant subspaces}, namely invariant subspaces for \( T \) that are invariant under every bounded operator in its commutant \( \{T\}' \).

\medskip

In this context, nevertheless, it remains unknown whether every subnormal operator, or more generally, every hyponormal operator has nontrivial closed hyperinvariant subspaces as long as it is not a multiple of the identity. Indeed, motivated by these results and open questions, the aim of this work is twofold. On one hand,  we prove that those hyponormal operators that can be represented as the sum of a tridiagonal operator and a trace-class operator have, indeed,  nontrivial closed hyperinvariant subspaces provided they are not multiple of the identity.

\medskip

On the other one, having in mind that Berger and Shaw Theorem \cite{bergershaw} yields that a cyclic hyponormal $T$ operator has trace-class self-commutator $[T^{*},T]$ (hence $T$ is essentially normal)  and moreover, if $T$ is a quasitriangular operator, the Brown-Douglas-Fillmore theorem implies $T$ can be written as the sum of a normal operator $N$ and a compact operator $K$, we determine when the operators $N+K$ are normal whenever $K$ is Hilbert-Schmidt. In particular, we exhibit non-normal hyponormal operators that can be written as compact perturbations of normal operators, what answers in the negative a question raised by Kim and Lee in \cite{kimlee} regarding an approach to solve the invariant subspace problem for hyponormal operators.

\medskip

The remainder of this paper is organized as follows. In Section 2, we begin by examining hyponormal tridiagonal operators and establishing the existence of nontrivial closed invariant subspaces for them. It is worth noting that, using moment sequences, Grivaux \cite{G02} proved the existence of invariant subspaces for positive tridiagonal operators on the classical Banach spaces $\ell^p$ ($1 \leq p < \infty$) and $c_0$. However, the general problem of determining whether tridiagonal operators admit nontrivial closed invariant subspaces remains open.

\medskip

Our approach will be proving that every non-normal hyponormal tridiagonal operator $T$ is not \emph{quasitriangular}. Recall that a bounded linear operator \( T \) on a separable infinite-dimensional Hilbert space \( H \) is said to be \emph{quasitriangular} if there exists an increasing sequence \( (P_n)_{n=1}^{\infty} \) of finite-rank projections converging strongly to the identity operator \( I \) as \( n \to \infty \), such that
\[
\|TP_n - P_nTP_n\| \to 0, \quad \text{as } n \to \infty.
\]
Building on the theorem of Aronszajn and Smith \cite{Aron1}, Halmos~\cite{Halmos} introduced the notion of quasitriangular operators in the 1960s to study the existence of invariant subspaces. In the seventies, Douglas and Pearcy~\cite{Douglas-Pearcy}, and Apostol, Foia\c{s}, and Voiculescu~\cite{AFV-1}, proved that those operators being not quasitriangular have nontrivial closed subspaces, so  the so-called \emph{Invariant Subspace Problem} is reduced to the class of quasitriangular operators (see Herrero’s monograph~\cite{Herrero} for further details).

\medskip

Indeed, we will prove that given a non-normal hyponormal tridiagonal operator $T$, there exists a complex number $\lambda\in\mathbb{C}$ such that $T-\lambda I$ is a noninvertible injective Fredholm operator. Consequently,  any non-normal hyponormal tridiagonal operator satisfies that the point spectrum $\sigma_p(T^{*})\neq\emptyset$, so every tridiagonal hyponormal operator has nontrivial closed hyperinvariant subspaces whenever it is not a multiple of the identity. Moreover, the same conclusion holds for those hyponormal operators having an \emph{eventually tridiagonal matrix} with respect to an orthonormal basis $\mathcal{E}$. Likewise, dropping-off the assumption on hyponormality, we show that  all those linear bounded operators $T$ having an \emph{eventually tridiagonal matrix} with respect to an orthonormal basis $\mathcal{E}$ have nontrivial hyperinvariant subspaces whenever their
self-commutator $[T^{*},T]$ is trace-class with nonzero trace.

\medskip

This sort of \emph{decomposition approach} will allow us then to show that  every hy\-po\-normal operator $T$ that can be expressed as a trace-class perturbation of a tridiagonal operator has nontrivial closed hyperinvariant subspaces provided $T$ is not a multiple of the identity.

\medskip

In Section 3 we discuss the normal behavior of those hyponormal operators that can be written as a compact perturbation of a normal operator and exhibit  non-normal hyponormal operators within such a class with trace class-commutator, which answers, in the negative, to a question raised by Kim and Lee  regarding an explicit approach  to the invariant subspace problem for hyponormal operators.

\medskip

Finally, in Section 4, we characterize the existence of reducing subspaces for hyponormal operators addressing the technique by Aronszajn and Smith in the setting of hyponormal operators and relating it to the class
$$\mathcal{S}_T:=\{Q\in\EL(\HH):(I-Q)TQ \text{ is self-adjoint}\}$$
introduced by the authors in \cite{clemgallard}.

\medskip

Throughout the rest of the paper, $\HH$ will always denote an infinite-dimensional separable complex  Hilbert space and $\EL(\HH)$ the Banach algebra of all bounded linear operators on $\HH$.

\section{Compact perturbations of tridiagonal operators: hyperinvariant subspaces}

In this section we will prove that every hyponormal operator that can be ex\-pressed as a trace-class perturbation of a tridiagonal operator has nontrivial closed hyperinvariant subspaces provided it is not a scalar multiple of the identity. As mentioned in the Introduction, our starting point will be examining hyponormal tridiagonal operators and establishing first the existence of nontrivial closed hyperinvariant subspaces for them.

\medskip

For such a goal, recall that $T\in \EL(\HH)$ is a \emph{tridiagonal operator} if there exists an (ordered) orthonormal basis $\mathcal{E}$ whose matrix representation with respect to  $\mathcal{E}$ is a tridiagonal matrix $(a_{mn})_{m, n\in \mathbb{N}}$. Accordingly, let  $\mathcal{E}= \{e_n\}_{n\in \mathbb{N}}$ denote an (ordered) orthonormal basis of $\HH$ fixed. If $\Lambda = \{\lambda_n\}_{n\in \mathbb{N}}$ is any bounded sequence in the complex plane $\C$,  the diagonal operator $D_\Lambda$ with respect to $\mathcal{E}$ associated to $\Lambda$ is defined by
$$D_\Lambda e_n = \lambda_n e_n,\qquad (n\geq 1).$$
Likewise, recall that the (unilateral) forward weighted shift $S_\Omega$  associated to a bounded complex sequence $\Omega=\{w_n\}_{n\in \mathbb{N}}$ is given by
$$S_\Omega e_n = w_n e_{n+1},\qquad (n\geq 1),$$
while its adjoint $S_\Omega^* e_n= \overline{w_{n-1}} e_{n-1}$ for every $n\geq 2$ and $S_\Omega^{*}e_1=0$. We refer to Shield's monograph \cite{Shields} as a basic reference for weighted shifts.

\medskip

\noindent Clearly, that if $T$ is a tridiagonal operator with respect to $\mathcal{E}$ with matrix $(a_{mn})_{m, n\in \mathbb{N}}$, then
$$T=D_\Lambda+S_\Omega+S_\Gamma^{*}$$
where $\Lambda=\{a_{nn}\}_{n\in\N}$, $\Omega=\{a_{n+1, n}\}_{n\in \mathbb{N}}$ and $\Gamma=\{\overline{a_{n,n+1}}\}_{n\in \mathbb{N}}$.

\medskip

With this notation, the next result turns out to be a simple computation.

\begin{proposition}
Let $T\in \EL(\HH)$ be a tridiagonal operator with respect to a basis $\mathcal{E}$ and matrix $(a_{mn})_{m, n\in \mathbb{N}}$ expressed by $T=D_\Lambda+S_\Omega+S_\Gamma^{*}$
where $\lambda_n = a_{nn}$, $w_n=a_{n+1, n}$, and $\gamma_n=\overline{a_{n,n+1}}$ for  $n\in \mathbb{N}$. Then $T$ is normal if and only if the following conditions hold
for every $n\in\N$
\begin{enumerate}
\item $|w_n|=|\gamma_n|$,
\item $\gamma_n(\lambda_{n+1}-\lambda_n)=w_n(\overline{\lambda_{n+1}}-\overline{\lambda_n})$,
\item $w_{n+1}\gamma_{n}=w_{n}\gamma_{n+1}$.
\end{enumerate}
\end{proposition}

Next proposition states that for non-normal hyponormal tridiagonal operators, the diagonal sequence in its matrix representation must be convergent as well as the modulus of the sequences determined by the superdiagonal and the subdiagonal, respectively.

\begin{proposition}\label{lemma1tridiagonal}
Let $T\in \EL(\HH)$ be a tridiagonal operator with respect to a basis $\mathcal{E}$ and matrix $(a_{mn})_{m, n\in \mathbb{N}}$ expressed by $T=D_\Lambda+S_\Omega+S_\Gamma^{*}$
where $\lambda_n = a_{nn}$, $w_n=a_{n+1, n}$ and $\gamma_n=\overline{a_{n,n+1}}$ for  $n\in \mathbb{N}$. Assume $T$ is a non-normal hyponormal operator. Then the sequence $\{\lambda_n\}_{n\in\N}$ converges. Moreover, the non-negative sequences $\{|w_n|\}_{n\in\N}$ and $\{|\gamma_n|\}_{n\in\N}$ are also convergent.
\end{proposition}

\medskip

\begin{proof}
First, let us show that the complex sequence $\{\lambda_n\}_{n\in\N}$ is convergent. Let us define the real numbers sequence $\{\alpha_n\}_{n\in\N_0}$ by:
\begin{equation}\label{sequence a_n}
\alpha_n=\left \{\begin{array}{ll}
|w_n|^2-|\gamma_n|^2& n\geq 1, \\
\noalign{\medskip}
0&n=0.
\end{array} \right.
\end{equation}
Note that
$$\langle [T^{*},T] e_n, e_n\rangle = |w_n|^2 + |\gamma_{n-1}|^2 - |w_{n-1}|^2-|\gamma_{n}|^2 = \alpha_n-\alpha_{n-1}\geq 0$$
for every $n\in \mathbb{N}$ and since, $T$ is assumed to be hyponormal,
$$\alpha_1\geq 0, \qquad \alpha_{n-1}\leq \alpha_n \qquad (n\in\N).$$

\medskip

Accordingly, $\{\alpha_n\}_{n\in\N_0}$ is a non-negative increasing sequence. Moreover, since $\alpha_n\leq 2\|T\|^2$ for every $n\in\N$, it converges implying that the series
\begin{equation}\label{trace-class}
\sum_{n\in\N}\langle[T^{*},T]e_n,e_n\rangle =\lim_n \alpha_n
\end{equation}
is also convergent. Consequently, $[T^{*},T]$ is a trace-class operator because it is positive. In particular, $\{\langle [T^{*},T]e_n,e_{n+1}\rangle\}_{n\in\N}$ is an $\ell^1$-sequence and hence
\begin{align*}
\langle [T^{*},T] e_n,e_{n+1}\rangle = \gamma_n(\lambda_{n}-\lambda_{n+1})+w_n(\overline{\lambda_{n+1}}-\overline{\lambda_{n}}) \in \ell^1.
\end{align*}

\smallskip

\noindent Since $T$ is assumed to be non-normal, there exists some $n_0\in\N$ verifying $\alpha_{n_0}>0$. Consequently, for every $n\geq n_0$,
$$\alpha_n=|w_n|^2-|\gamma_n|^2=(|w_n|+|\gamma_n|)(|w_n|-|\gamma_n|)\geq \alpha_{n_0}>0$$
with $|w_n|,|\gamma_n|\leq \|T\|$. Hence, for every $n\geq n_0$,
$$|w_n|-|\gamma_n|\geq \displaystyle \frac{\alpha_{n_0}}{2\|T\|}>0,$$
and
\begin{eqnarray*}
\left|\langle [T^{*},T] e_n,e_{n+1}\rangle\right|&= &\left |\gamma_n(\lambda_{n}-\lambda_{n+1})+w_n(\overline{\lambda_{n+1}}-\overline{\lambda_n})\right | \\
& \geq & \left||w_n|-|\gamma_n|\right||\lambda_{n+1}-\lambda_n| \\
& \geq & \displaystyle  \frac{\alpha_{n_0}}{2\|T\|}|\lambda_{n+1}-\lambda_n|.
\end{eqnarray*}
Consequently,  $\{\lambda_{n+1}-\lambda_n\}_{n\in\N}$ is an $\ell^1$-sequence, which implies that $\{\lambda_n\}_{n\in\N}$ must be convergent, as claimed.

\medskip

Let us prove now that the non-negative sequences $\{|w_n|\}_{n\in\N}$ and $\{|\gamma_n|\}_{n\in\N}$ must be also convergent.

\medskip

Since \eqref{trace-class} yields that the self-commutator $[T^{*},T]$ is trace-class, it follows that $\{\langle [T^{*},T]e_n,e_{n+2}\rangle \}_{n\in\N}$ must be an $\ell^1$-sequence, that is,
$$\langle [T^{*},T]e_n,e_{n+2}\rangle = \gamma_{n+1}w_{n}-w_{n+1}\gamma_{n}\in \ell^1.$$
Multiplying by the bounded sequence $\{\gamma_{n+1}w_{n}+w_{n+1}\gamma_{n}\}_{n\in\N}$, it follows that the sequence
$\{\gamma_{n+1}^2w_{n}^2 -w_{n+1}^2\gamma_{n}^2\}_{n\in \mathbb{N}} \in \ell^1.$ Now,
\begin{eqnarray*}
|\gamma_{n+1}^2w_{n}^2 -w_{n+1}^2\gamma_{n}^2|&\geq & \left | |\gamma_{n+1}|^2|w_{n}|^2 -|w_{n+1}|^2 | \gamma_{n}|^2 \right|\\
&=& \left||\gamma_{n+1}|^2(\alpha_{n}-\alpha_{n+1})+(|\gamma_{n+1}|^2-|\gamma_{n}|^2)\alpha_{n+1}\right|,
\end{eqnarray*}
and the sequence $\{|\gamma_{n+1}|^2(\alpha_{n}-\alpha_{n+1})\}_{n\in\N}$ is summable, so $$\{(|\gamma_{n+1}|^2-|\gamma_{n}|^2)\alpha_{n+1}\}_{n\in\N}\in\ell^1.$$

\smallskip

Since $\alpha_{n}\geq \alpha_{n_0}>0$ for every $n\geq n_0$, it follows that $\{|\gamma_{n+1}|^2-|\gamma_{n}|^2\}_{n\in\N}\in\ell^1$, that is, $\{|\gamma_n|\}_{n\in\N}$ is convergent. Since $\{\alpha_n\}_{n\in\N}$ also converges, $\{|w_n|\}_{n\in\N}$ must converge as well, as we wished to prove.
\end{proof}

The next theorem establishes that every non-normal hyponormal tridiagonal operator is not quasitriangular. As mentioned in the Introduction,  an operator $T \in \EL(\HH)$ is called \emph{quasitriangular} if there exists an increasing sequence of finite-rank projections $(P_n)$ converging strongly to the identity such that $\|TP_n - P_nTP_n\| \to 0$.

\begin{theorem}\label{theo1tridiagonal}
Let $T\in \EL(\HH)$ be a tridiagonal operator with respect to a basis $\mathcal{E}$.  Assume $T$ is hyponormal. Then $T$ is normal if and only if $T$ is quasitriangular.
\end{theorem}

\medskip

In order to prove Theorem \ref{theo1tridiagonal}, we will make use of the close connection between quatriangularity and the Calkin algebra showed by Apostol, Foias and Voiculescu, and Douglas and Pearcy in the seventies. If $\mathcal{K}(\HH)$ denotes the ideal of compact operators in $\EL(\HH)$, recall that the \emph{Calkin algebra} is the quotient $\mathcal{C}(H) = \EL(\HH)/\mathcal{K}(\HH)$, representing operators modulo compact perturbations. An operator $T \in \EL(\HH)$ is \emph{semi-Fredholm} if its range is closed and either the kernel of $T$, $\ker T$, or the kernel of $T^{*}$ is finite-dimensional; and it is called \emph{Fredholm} if both $\ker T$ and $\ker T^{*}$ are finite-dimensional; or equivalently, $T$ is Fredholm if and only if its class $[T]$ is invertible in $\mathcal{C}(H)$. The \emph{Fredholm index} for a semi-Fredholm operator is then defined by
$$\operatorname{ind}(T) = \dim(\ker T) - \dim(\ker T^{*})$$
an integer invariant-stable under compact perturbations.

\medskip

In particular, Douglas and Pearcy \cite[Theorem 1]{Douglas-Pearcy} and Apostol, Foias and Voiculescu \cite{AFV-1} showed that $T$ is quasitriangular its \textit{negative semi-Fredholm domain is empty} (see  \cite[Chapter6]{Herrero} for more on the subject).
Equivalently, $T$ is quasitriangular if and only if $T-\lambda I$ is not semi-Fredholm with negative index, that is,\
$$\operatorname{ind}(T - \lambda I) \ge 0$$
for all $\lambda \in \mathbb{C}$ such that $T - \lambda I$ is semi-Fredholm. Thus, the Calkin algebra provides a natural framework for characterizing quasitriangular operators via the Fredholm index.

\medskip

\medskip

\begin{demode}\emph{Theorem \ref{theo1tridiagonal}. }
Since every normal operator is quasitriangular, in order to prove the nontrivial statement of the theorem, let us assume that $T$ is not normal. Since $T$ is assumed to be a tridiagonal operator with respect to the basis $\mathcal{E}$, denote $\mathcal{E}=\{e_n\}_{n\in \mathbb{N}}$ and $(a_{mn})_{m, n\in \mathbb{N}}$ the matrix representation of $T$ and write
$$T=D_\Lambda+S_\Omega+S_\Gamma^{*}$$
where $\lambda_n = a_{nn}$, $w_n=a_{n+1, n}$ and $\gamma_n=\overline{a_{n,n+1}}$ for  $n\in \mathbb{N}$. Note that $a_{n+k, n}=a_{n, n+k}=0$ for all $k\geq 2$.

\medskip

Proposition \ref{lemma1tridiagonal} along with a suitable application of a unitary transformation (of the form $Ue_n = s_n e_n$  with $|s_n|=1$ for every $n\in\N$) allows us to assume the following:
\begin{enumerate}
\item The sequence $\{\lambda_n\}_{n\in\N}$ converges to some $\lambda\in\C$,
\item the sequence $\{w_n\}_{n\in\N}$ is non-negative and convergent to some $w>0$,
\item the sequence $\{|\gamma_n|\}_{n\in\N}$ converges to some $\gamma\geq 0$, with $w>\gamma$.
\end{enumerate}

Moreover, by means of Proposition \ref{lemma1tridiagonal},  $\{\gamma_{n+1}w_{n}-w_{n+1}\gamma_{n}\}_{n\in\N}$ is an $\ell^1$ sequence, so it converges to zero, implying that the sequence $\{\gamma_{n+1}-\gamma_{n}\}_{n\in\N}$ converges to zero as well.

\medskip

We proceed by distinguishing two cases.

\medskip

\noindent $\bullet$ \underline{Case 1}: $\gamma=0$.

\medskip

In this case, $T$ is essentially unitarily equivalent (namely, unitarily equivalent in the Calkin algebra) to the operator $T_2=\lambda I +w S$ (where $S$ is the unilateral shift). Accordingly, $T-\lambda I$ is Fredholm with $\text{ind}(T_2-\lambda)=\text{ind }wS=-1$ and the result holds.

\medskip

\noindent $\bullet$ \underline{Case 2}: $\gamma>0$.

\medskip

In this case $T$ is essentially equivalent to the operator $T_2=\lambda I+w S+S^{*}_{\Gamma}$.

\smallskip

Now, since $[S,S_\Gamma^{*}]$ is a diagonal operator with diagonal sequence $\{\overline{\gamma_{n-1}}-\overline{\gamma_{n}}\}_{n\in\N}$ (where $\gamma_0=0$), it is compact. Thus the operators $\lambda I+S$ and $S_\Gamma^{*}$ essentially commute and we have that the essential spectra of $T-\lambda I$ satisfies
$$\sigma_e(T-\lambda I)\subseteq  \sigma_e(wS)+\sigma_e(S^{*}_{\Gamma})=\{we^{i\theta}+ \gamma e^{i\delta}:\theta,\delta\in[0,2\pi]\}.$$
Since $w>\gamma$, it follows that $0\notin\{we^{i\theta}+ \gamma e^{i\delta}:\theta,\delta\in[0,2\pi]\}$ and hence, $T-\lambda I$ is a Fredholm operator.

Finally, having into account that $\text{ind}(T-\lambda I)=\text{ind}(T_2-\lambda I)=-1$, it follows that $T$ is not quasitriangular, and Theorem \ref{theo1tridiagonal} is proved.
\end{demode}

A straightforward consequence of Theorem \ref{theo1tridiagonal} is the following;

\begin{corollary}\label{cor1tridiagonal}
Every hyponormal tridiagonal operator $T\in \EL(\HH)\setminus \mathbb{C}\Id$ has nontrivial closed hyperinvariant subspaces.\end{corollary}

Next corollary shows that it is possible to drop off the assumption on hyponormality in Corollary \ref{cor1tridiagonal} for operators having an \emph{eventually tridiagonal matrix} with respect to an orthonormal basis $\mathcal{E}$ with a pay-off: The self-commutator $[T^{*},T]$ must be a trace-class operator.

\medskip

\begin{corollary}\label{cor3tridiagonal}
Let $T\in \EL(\HH)\setminus \mathbb{C}\Id$  and suppose that there exists an orthonormal basis $\mathcal{E}=\{e_n\}_{n\in\N}$ such that
$$\langle T e_n, e_{n+ \ell}\rangle = \langle T^{*} e_n, e_{n + \ell}\rangle =0$$
for every $\ell\geq 2$ and $n\geq n_0$ for some $n_0\in\N$. If $[T^{*},T]$ is a trace-class operator with non-zero trace, then $T$ has nontrivial closed hyperinvariant subspaces.
\end{corollary}

\medskip

\begin{proof}
Note that if $n_0=1$, $T$ is a tridiagonal operator and if $n_0\geq 2$, the matrix representation of $T$ with respect to the decomposition  $\HH=\Span\{e_1,\dots,e_{n_0-1}\}\oplus\overline{\Span\{e_n:n\geq n_0\}}$ is given by:
$$T=\begin{pmatrix}
A & B\\
C & \tilde{T}_\Lambda
\end{pmatrix},$$
where $A,B,C$ are finite-rank operators and $\tilde{T_\Lambda}$ is a tridiagonal operator in
$$\overline{\Span \{e_n:n\geq n_0\}}$$
which can be expressed by $\tilde{T_\Lambda}=D_\Lambda + S_\Omega+S_{\Gamma}^{*}$, as in the proof of Corollary \ref{cor1tridiagonal}.

\smallskip

In this latter case, the self-commutator of $T$ is
$$[T^{*},T]=\begin{pmatrix}
A^{*}A-AA^{*}+C^{*}C-BB^{*} & *\\
* & [\tilde{T}_\Lambda^{*},\tilde{T}_\Lambda]+B^{*}B-CC^{*}.
\end{pmatrix}$$

Therefore, its trace turns out to be
$$\text{tr}([T^{*},T])=\text{tr}([A^{*},A])+\left[\begin{pmatrix}0 & B\\
C& 0 \end{pmatrix}^{*},\begin{pmatrix}0 & B\\
C& 0 \end{pmatrix}\right]+\text{tr}([\tilde{T}_\Lambda^{*},\tilde{T}_\Lambda])=\text{tr}([\tilde{T}_\Lambda^{*},\tilde{T}_\Lambda])$$
where we note that $A,B,C$ are finite-rank operators and the traces of their self-commutators are zero.

\smallskip

So, in any case, $[\tilde{T}_\Lambda^{*},\tilde{T}_\Lambda]$ is a trace-class operator since $[T^{*},T]$ is  trace-class.

In the case $\text{tr}[T^{*},T]>0$, arguing as in Proposition \ref{lemma1tridiagonal} and Theorem \ref{theo1tridiagonal}, it follows that $\{\lambda_n\}_{n\in\N}$ is a convergent sequence to a complex $\lambda$ such that $T-\lambda I$ is a Fredholm operator with $\text{ind}(T-\lambda I)=-1$, which yields the statement of the corollary.

In the case $\text{tr}[T^{*},T]<0$ we can consider the adjoint operator

$$T^{*}=\begin{pmatrix}
A^{*} & C^{*}\\
B^{*} & \tilde{T}_\Lambda^{*}
\end{pmatrix},$$

\noindent that verifies the same hypothesis than $T$ and the trace of its self-commutator is strictly positive. By the previous case, there exists some $\overline{\lambda}\in\C$ such that $\{\overline{\lambda_n}\}_{n\in\N}$ is convergent to $\overline{\lambda}$ such that $T^{*}-\overline{\lambda} I$ is Fredholm with $\text{ind}(T^{*}-\overline{\lambda} I)=-1$, so $T-\lambda I$ is Fredholm and $\text{ind}(T-\lambda I)=1$. This completes the proof.
\end{proof}

\medskip

The previous results yield a \emph{decomposition approach} in order to show the main result of this section which states that hyponormal operators $T\in \EL(\HH)\setminus \mathbb{C}\Id$ that admit a decomposition \( T = R + V \), where \( R \) is tridiagonal and \( V \) is trace-class have nontrivial closed hyperinvariant subspaces.

\medskip

\begin{theorem}
Let $T\in \EL(\HH)\setminus \mathbb{C}\Id$ a hyponormal operator and suppose that there exists an orthonormal basis $\mathcal{E}$ such that \( T = R + V \) with $R$ a tridiagonal operator with respect to $\mathcal{E}$ and $V$ a trace-class operator. Then $T$ has nontrivial closed hyperinvariant subspaces.
\end{theorem}

\medskip

\begin{proof}
Assume $T$ is not normal and write the tridiagonal operator $R=D_\Lambda+S_\Omega+S_\Gamma^{*}$ with $\Lambda=\{\lambda_n\}_{n\in \mathbb{N}}$, $\Omega=\{w_n\}_{n\in \mathbb{N}}$ and $\Gamma=\{\gamma_n\}_{n\in \mathbb{N}}$ the sequences related to the diagonal, subdiagonal and superdiagonal entries of the matrix of $R$.

Recalling the sequence $\{\alpha_n\}_{n\in\N}$ in \eqref{sequence a_n}:
\begin{equation*}
\alpha_n=\left \{\begin{array}{ll}
|w_n|^2-|\gamma_n|^2& n\geq 1, \\
\noalign{\medskip}
0&n=0,
\end{array} \right.
\end{equation*}
we have
$$\langle [T^{*},T]e_n,e_n\rangle = \alpha_n-\alpha_{n-1} + \langle ([R^{*},V]+[V^{*},R]+[V^{*},V])e_n,e_n\rangle.$$

Now, since the operator $V$ is trace-class, it holds  $([R^{*},V]+[V^{*},R]+[V^{*},V])$ is a trace-class operator with zero trace (see \cite{laurictracezero}).

Then
$$\sum_{n=1}^\infty \langle [T^{*},T]e_n,e_n\rangle =\lim_N\sum_{n=1}^N \langle [T^{*},T]e_n,e_n\rangle =\lim_N\sum_{n=1}^N (\alpha_n-\alpha_{n-1})=\lim_N \alpha_N.$$

And, since the sequence $\{\alpha_n\}_{n\in\N}$ is bounded, it holds the above limit is finite, so $[T^{*},T]$ is a trace-class operator. This implies $[R^{*},R]$ is trace-class.

Consequently, $[R^{*},R]$ is also a trace-class operator with $\text{tr}([T^{*},T])=\text{tr}([R^{*},R])>0$, and since $R$ is tridiagonal, Corollary \ref{cor3tridiagonal} implies that there exists a complex number $\lambda$ such that $\text{ind}(R-\lambda I)=-1$. Consequently, $\text{ind}(T-\lambda I)=-1$ and therefore $T$ has nontrivial closed hyperinvariant subspaces.
\end{proof}

Using this last theorem we can deduce the existence of hyperinvariant subspaces for those hyponormal operators having an \emph{eventually tridiagonal matrix} with respect to an orthonormal basis $\mathcal{E}$:

\begin{corollary}\label{cor2tridiagonal}
Let $T\in \EL(\HH)\setminus \mathbb{C}\Id$ be a hyponormal operator. Assume that there exists an orthonormal basis $\mathcal{E}=\{e_n\}_{n\in\N}$ such that
$$\langle T e_n, e_{n+ \ell}\rangle = \langle T^{*} e_n, e_{n+ \ell}\rangle =0$$
for every $\ell\geq 2$ and $n\geq n_0$ for some $n_0\in\N$. Then $T$ has  nontrivial closed hyperinvariant subspaces.
\end{corollary}

\medskip

\begin{proof}
Note that if $n_0=1$, $T$ is a tridiagonal operator, and if $n_0\geq 2$, the matrix representation of $T$ with respect to the decomposition  $\HH=\Span\{e_1,\dots,e_{n_0-1}\}\oplus\overline{\Span\{e_n:n\geq n_0\}}$ is given by:

$$T=\begin{pmatrix}
A & B\\
C & \tilde{T}_\Lambda
\end{pmatrix},$$
so $T=R+V$ with
$$R=\begin{pmatrix}
0 & 0\\
0 & \tilde{T}_\Lambda
\end{pmatrix} \quad  V= \begin{pmatrix}
A & B\\
C & 0
\end{pmatrix}.$$

Then we can apply the previous result since $R$ is a tridiagonal operator and $V$ is trace-class (notice $A$,$B$ and $C$ are finite-rank operators).
\end{proof}

\medskip

\begin{remark}
It is worthy to point out that given any linear bounded operator $T$ with a cyclic vector $x$, upon applying the  Gram-Schmidt procedure to the sequence $\{T^n x:n\geq 0\}$ yields an orthonormal basis $\{e_n\}_{n\in\N}$ of $\HH$ such that $T$ admits a matrix representation ``almost upper triangular'', namely,
$$
\langle T e_n, e_{n+k} \rangle=0 \text{ for all } k\geq 2,
$$
(see \cite{RR}). In particular, $T$ can be expressed as the sum of a tridiagonal operator $R$ and a linear bounded operator $V$ whose matrix representation is upper triangular with zero entries in both the main diagonal and superdiagonal. Consequently, if $T$ is hyponormal and $V$ trace-class, $T$ has nontrivial closed hyperinvariant subspaces.
\end{remark}

\medskip

\subsection{Point spectrum of hyponormal tridiagonal operators}
Note that previous results ensure that for every tridiagonal operator $T$ which is  non-normal hyponormal  (or even for every \emph{eventually tridiagonal} operator being non-normal hyponormal), there exists some complex $\lambda$ in the point spectrum of $T^{*}$. Moreover, it holds that the range $\ran(T-\lambda I)$ is closed and
$$\dim\ker(T^{*}-\overline{\lambda} I)=\dim\ker(T-\lambda I)+1.$$

\smallskip

However, we wonder if $\ker(T-\lambda I)$ is also nontrivial. The next result shows that this subspace is always trivial whenever the subdiagonal of the matrix representation of $T$ has no zero entries, and in such a case, clearly $\ker(T^{*}-\overline{\lambda} I)$ is one-dimensional.

\smallskip

\begin{proposition}
Let $T\in \EL(\HH)$ be a tridiagonal operator with respect to a basis $\mathcal{E}$ and let $(a_{mn})_{m, n\in \mathbb{N}}$ be the matrix representation of $T$. Assume $T$ is a non-normal hyponormal operator. If $a_{n+1, n}\neq 0$ for every $n\in\N$, then the point spectrum $\sigma_p(T)=\emptyset$.
\end{proposition}

\medskip

\begin{proof}
We will argue by contradiction. But, first, we deal with some preliminaries. Let $\mathcal{E}=\{e_n\}_{n\in \mathbb{N}}$ and write, as usual,
$$T=D_\Lambda+S_\Omega+S_\Gamma^{*}$$
where $\lambda_n = a_{nn}$, $w_n=a_{n+1, n}$ and $\gamma_n=\overline{a_{n,n+1}}$ for  $n\in \mathbb{N}$.

\smallskip

A proof by induction yields that there exists a sequence of complex polynomials $\{p_m(z)\}_{m\in\N}$ such that
$$\left(\prod_{j=1}^{m-1}w_j\right) e_m=p_m(T)e_1 \quad m\geq 2$$
with
$$p_{m+1}(z)=(z-\lambda_m)p_m(z) - \overline{\gamma_{m-1}}w_{m-1}p_{m-1}(z) \quad m\geq 2 $$
$$p_1(z)=1, \quad p_2(z)=z-\lambda_1.$$

\smallskip

Since $T^{*}$ is also a tridiagonal operator, there exists a sequence $\{q_m(z)\}_{m\in\N}$ of polynomials such that
$$\left(\prod_{j=1}^{m-1}\gamma_j\right) e_m = q_m(T^{*})e_1 \quad m\geq 2$$
with
$$q_{m+1}(z)=(z-\overline{\lambda_m})q_m(z) - \overline{w_{m-1}}\gamma_{m-1}q_{m-1}(z) \quad m\geq 2 $$
$$q_1(z)=1, \quad p_2(z)=z-\overline{\lambda_1}.$$

\smallskip

For the sake of simplicity, given any complex polynomial $r(z)=\sum_j \beta_j z^j$, let us denote $\overline{r}(z)=\sum_j \overline{\beta_j}z^j$. Clearly
$q_m(z)=\overline{p_m}(z)$  for every $m\in\N$.

\smallskip

Now, let us assume that $\sigma_p(T)\neq \emptyset$. Then $Tx=\lambda x$ for some $\lambda\in\C$ and $x\neq 0$ and, in addition, $T^{*}x=\overline{\lambda} x$ since $T$ is hyponormal.

\smallskip

Without loss of generality, we may suppose that $\lambda \neq 0$. Indeed, if $Tx=0$, it holds that $(T+\alpha I)x=\alpha x$ for every $\alpha\neq 0$ and $T+\alpha I$ is a hyponormal tridiagonal operator with (obviously) non-zero entries in the subdiagonal of its matrix representation.

\smallskip

Write $x=\sum_n x_n e_n$ and consider the sequences:
$$b_n=\left \{\begin{array}{ll}
1 & n=1\\
\prod_{j=1}^{n-1} w_j & n\geq 2,
\end{array} \right. $$
and
$$c_n=\left \{\begin{array}{ll}
1 & n=1\\
\prod_{j=1}^{n-1} \gamma_j & n\geq 2.
\end{array} \right. $$

\smallskip

Note that
$$\overline{b_n} x_n=\overline{b_n}\langle x,e_n\rangle =\langle x,p_n(T)e_1\rangle=\langle \overline{p_n}(T^{*})x,e_1\rangle =\overline{p_n}(\overline{\lambda})x_1 \quad n\in\N,$$
$$\overline{c_n} x_n=\overline{c_n} \langle x,e_n\rangle = \langle x, \overline{p_n}(T^{*})e_1\rangle=\langle p_n(T)x,e_1\rangle=p_n(\lambda)x_1 \quad n\in\N.$$

Since $T$ is a non-normal operator, arguing as in Proposition \ref{lemma1tridiagonal}, there exists $n_0\in\N$ such that $\alpha_{n_0}>0$ and $\alpha_n\geq \alpha_{n_0}>0$ for every $n\geq n_0$, where $\alpha_n$ is the sequence given in \eqref{sequence a_n}. Then $|b_n|>|c_n|$ for every $n\geq n_0$. Since $|p_n(z)|=|\overline{p_n}(\overline{z})|$ for every $n\in\N$ and every $z\in\C$, it holds that $x_n=0$ for every $n\geq n_0$.

\smallskip

Then, if $x_{n_0-1}\neq 0$, it follows that
$$\langle Tx,e_{n_0}\rangle = x_{n_0-1}w_{n_0-1}\neq 0$$
but it must hold $\langle Tx,e_{n_0}\rangle  =\lambda x_{n_0}=0$ (and we had $\lambda\neq 0$). Clearly, this is a contradiction. Proceeding by induction we get $x_n=0$ for every $n\in\N$ and hence, $x$ is the zero vector. Accordingly, $\sigma_p(T)=\emptyset$, as claimed.
\end{proof}

\section{Hyponormal operators as compact perturbations of normal operators}

As stated in the Introduction, when studying invariant subspaces of a hyponormal operator \( T \), we may, without loss of generality, assume that \( T \) can be expressed as \( T = N + K \), where \( N \) is normal and \( K \) is compact. Moreover, by Berger and Shaw theorem, we may also assume that the self-commutator $[T^{*},T]$ is trace-class. Accordingly, it turns out to be natural to force the compact operator to have some growth condition over its singular values in order to obtain properties over the hyponormal operator $T$.
For example, $T$ is necessarily normal when $K$ is in the Hilbert-Schmidt class. In particular, the following proposition was proved by Lauric and Pearcy (see \cite{lauricpearcy}).

\begin{proposition}
A hyponormal operator $T$ is normal if and only if it can be written as $T=D+K$ where $D$ is a diagonal operator and $K$ is a Hilbert-Schmidt operator.
\end{proposition}

\begin{proof}
The necessity is just a consequence of a remarkable result by Voiculescu \cite{voiculescu} (see also \cite{davidson}), which states that every normal operator can be expressed as the sum of a diagonal operator and a Hilbert-Schmidt one.

For the converse, let us assume that $T=D_\Lambda + K$ where $D_\Lambda$ is a diagonal operator with respect to an orthonormal basis $\{e_n\}_{n\in\N}$ and $K$ is a Hilbert-Schmidt operator. The following argument relies on one given in \cite{kimlee}.

Note that
\begin{align*}
[T^{*},T]&=(D_\Lambda+K)^{*}(D_\Lambda+K)-(D_\Lambda+K)(D_\Lambda+K)^{*}=\\
&=(K^{*}D_\Lambda-D_\Lambda K^{*})+(D_\Lambda^{*}K-K D_\Lambda^{*})+[K^{*},K].
\end{align*}

Let $\{\lambda_n\}_{n\in\N}$ be the sequence of complex eigenvalues of $D_\Lambda$ with respect to the orthonormal basis $\{e_n\}_{n\in\N}$. Then
$$\langle [T^{*},T]e_n,e_n\rangle = \langle [K^{*},K] e_n,e_n\rangle.$$

Since the operator $K$ is Hilbert-Schmidt, it follows that both operators $K^{*}K$ and $KK^{*}$ are trace class and they have equal trace. This implies
$$\sum_{n\in\N} \langle[T^{*},T]e_n,e_n\rangle = \text{tr}(K^{*}K)-\text{tr}(KK^{*})=0$$
and, since $[T^{*},T]\geq 0$, we have $[T^{*},T]=0$. Accordingly, $T$ is normal, as we claimed.
\end{proof}

As a consequence, we have the following:

\begin{corollary}\label{corpolynomialhyponormal}
Let $T\in \EL(\HH)\setminus \mathbb{C}\Id$  and suppose $T=N+K$ with $N$ normal and $K$ a Hilbert-Schmidt operator. Assume further that there exists some polynomial $p(t)$ such that $p(T)$ is hyponormal. Then $T$ has nontrivial closed  hyperinvariant subspaces.
\end{corollary}

Before proving Corollary \ref{corpolynomialhyponormal}, it is worthy to point out that the class of polynomially hyponormal operators has called the attention of operator theorists in the last decades. Recall that an operator $T\in\EL(\HH)$ is called \textit{polynomially hyponormal} if $p(T)$ is hyponormal for every complex polynomial $p(z)$. It was proved by B. Prunaru that every polynomially hyponormal operator has nontrivial closed invariant subspaces. We refer to \cite{curtoputinar} for further properties and references regarding polynomially hyponormal operators.

\medskip

\medskip

\begin{demode}\emph{Corollary \ref{corpolynomialhyponormal}.}
Note that $p(T)=p(D)+K'$ where $D$ is diagonal and $K'$  is Hilbert-Schmidt. So $p(T)$ is normal. In the case $p(T)$ is a scalar multiple of the identity, then $T$ has eigenvalues. On the other hand, if $p(T)$ is not a multiple of the identity operator, the result follows from the fact $\{T\}'\subseteq\{p(T)\}'$.
\end{demode}

\medskip

 The next result shows that if $T$ is hyponormal and $T=N+K$ with both $N, K$  normal operators, and $K$ compact, then $T$ is indeed normal:

\begin{proposition}
Let $T\in \EL(\HH)$ a hyponormal operator and suppose $T=N+K$ with $N, K$ normal operators and $K$ compact. Then $T$ is normal.
\end{proposition}

\begin{proof}
Let $\{e_n\}_{n\in\N}$ be an orthonormal basis of eigenvectors of $K$. By computing $\langle [T^{*},T]e_n,e_n\rangle$, it follows that
$$\langle [T^{*},T]e_n,e_n\rangle=\langle [N^{*},N]e_n,e_n\rangle \quad n\in\N.$$
Since $N$ is normal, it holds that $[T^{*},T]=0$, and hence, $T$ is also normal.
\end{proof}

\medskip

\medskip

These results motive an strategy for the invariant subspace problem of hyponormal operators included in Question 4.1 \cite{kimlee}:

Suppose $T\in\EL(\HH)$ satisfy the following:
\begin{enumerate}
    \item $T$ can be written as a compact perturbation of a diagonal operator;
    \item $T$ is hyponormal;
    \item $[T^{*},T]$ is a trace-class operator.
\end{enumerate}

\noindent Kim and Lee asked the following:

\medskip

\medskip

\begin{quote}
Does it follow that either $T$ is normal or $[K^{*},K]$ is a trace-class operator, where $T=D+K$ for $D$ diagonal and $K$ compact?
\end{quote}

\medskip

The answer to this question is, in general, negative. Indeed,  first note that any hyponormal operator verifying the conditions (1)-(3) above such that $[K^{*},K]$ is trace-class is necessarily normal (see, for instance, \cite[Proposition 4.2]{kimlee} or \cite[Proposition 1.4]{laurictracezero}).

Accordingly, given any non-normal hyponormal operator $R$ with trace-class self-commutator, we can consider a sequence $\{\lambda_n\}_{n\in\N}$ dense in $\sigma(R)$ and a diagonal operator $D$ with eigenvalues $\{\lambda_n\}_{n\in\N}$ such that $\sigma(D)=\sigma_e(D)$ and define the operator
$$T=D\oplus R$$
on the Hilbert space $\HH\oplus\HH$. Note that such a $T$ is a non-normal hyponormal operator with spectrum $\sigma(T)=\sigma_e(T)= \sigma(R)$. Moreover, it holds that  $$\text{ind}(R-\lambda I)=0$$
for every $\lambda\in\C$ such that $T-\lambda\Id$ is a semi-Fredholm operator, so, since $T$ is essentially normal, it can be written as a compact perturbation of a diagonal operator by the Brown-Douglas-Fillmore theorem (see \cite{bdftheorem}).

\medskip

As concrete examples, we may consider $R$  be a non-normal hyponormal Toeplitz operator $T_\varphi$ in the Hardy space $H^2$ with compact self-commutator (e.g., with $\varphi\in H^\infty$, namely, $\varphi$ belonging to the Banach space consisting of all bounded analytic functions on the open unit disk $\mathbb{D}$
endowed with the supremum norm). Hyponormal Toeplitz operators were characterized by Cowen in \cite{cowen}.

\medskip

In the case we also look for $[T^{*},T]$ to be trace-class, we have to require  the Fourier coefficients $\{\hat{\varphi}(n)\}_{n\in\N}$ of $\varphi\in H^\infty$ to verify
$$\sum_{n\in\N}n|\hat{\varphi}(n)|^2<\infty.$$

\section{Reducing subspaces for hyponormal operators}

The previous sections show a clear dichotomy regarding the classes of hyponormal operators considered: either normality or non-quasitriangularity. In particular, all quasitriangular hyponormal operators $T$ studied before are reducible (i.e., there exists a nontrivial closed invariant subspace for both $T$ and $T^*$). Motivated by this fact, at this point, we are interested in characterizing the existence of reducing subspaces for hyponormal operators.

\medskip

Given an operator $T\in\EL(\HH)$, the authors introduced in \cite{clemgallard}  the class of operators
$$\mathcal{S}_T:=\{Q\in\EL(\HH):(I-Q)TQ \text{ is self-adjoint}\}$$
in order to study the existence of nontrivial closed invariant subspaces for compact perturbations of self-adjoint operators.

\medskip

This class arises naturally through the application of the Aronszajn–Smith procedure in the study of quasitriangular operators. In particular, given $T\in\EL(\HH)$ and $Q\in\mathcal{S}_T$ self-adjoint, it is straightforward that $\ker(I-Q)$ is a  (non-necessarily nontrivial) closed invariant subspace for $T$ and analogously $\ker Q$ for $T^{*}$. Moreover, if $P\in\mathcal{S}_T$ is a projection, then $(I-P)TP=0$ and $P(\HH)$ is an invariant subspace for $T$. Using this class $\mathcal{S}_T$, the authors proved that the \emph{Invariant Subspace Problem} is equivalent to show that the class $\mathcal{S}_T$ is nontrivial (in the sense that it contains more operators than just $0$ and $I$) for every $T\in\EL(\HH)$ (see \cite{clemgallard}).

\medskip

Similarly, the class $\mathcal{S}_T$ offers an alternative characterization of when a given operator $T \in \EL(\HH)$ admits nontrivial closed reducing subspaces. In particular, the existence of reducing subspaces $M$ for operator $T$ such that the imaginary part of $T$ is positive and vanishes on $M$ (studied in \cite{foias})  is completely characterized by the existence of a self-adjoint operator $Q\in\mathcal{S}_T$ with an eigenvalue different from $0$ and $1$ (see \cite{clemgallard}).
\medskip

In this section, we will show that this class plays also a key role in characterizing the existence of reducing subspaces for hyponormal operators. For such a goal, let us recall that any operator $T\in\EL(\HH)$ can be written as $T=X+iY$ where both $X$ and $Y$ are self-adjoint operators (called the real part and imaginary part of $T$, respectively):
$$X=\frac{T+T^{*}}{2} \quad Y=\frac{T-T^{*}}{2i}.$$
In what follows, we will use the notation $T^n=X_n+iY_n$ for $T\in\EL(\HH)$ and $n\in\N$, where both $X_n$ and $Y_n$ are the real and the imaginary parts of $T^n$, respectively.

\medskip

We start noting that, for a self-adjoint operator \( Q \), the condition \( Q \in \mathcal{S}_T \) can be expressed as
\[
TQ = QT^{*} + 2iQY_1Q,
\]
or, equivalently,
\[
QX_1 - X_1Q = i(QY_1 + Y_1Q - 2QY_1Q).
\]
The following result sets the stage for the proof of the main theorem of this section.

\begin{lemma}\label{lemma1}
Let $T\in\mathcal{L}(\mathcal{H})$ and let $Q\in\mathcal{S}_T$ be self-adjoint. If $QX_1=X_1Q$, then $Q\in \mathcal{S}_{X_1+i[T^{*},T]}$.
\end{lemma}

\begin{proof}
Since $Q\in\mathcal{S}_T$, it holds that

\begin{eqnarray*}
QT^{*}T&=&TQT-2iQY_1QT \\
&=&T(TQ+2iQY_1-2iQY_1Q)-2iQY_1(TQ+2iQY_1-2iQY_1Q)\\
&=&T^2Q+2iTQY_1-2iTQY_1Q-2iQY_1TQ+4QY_1QY_1-4QY_1QY_1Q,\\
\noalign{\medskip}
\noalign{and}
\noalign{\medskip}
QTT^{*}&=&(TQ+2iQY_1-2iQY_1Q)T^{*}\\
&=& T(TQ-2iQY_1Q)+2iQY_1T^{*}-2iQY_1(TQ-2iQY_1Q)\\
&=&T^2Q-2iTQY_1Q+2iQY_1T^{*}-2iQY_1TQ-4QY_1QY_1Q.
\end{eqnarray*}

\medskip

\noindent Likewise,
\begin{eqnarray*}
T^{*}TQ=(QT^{*}T)^{*}&=&QT^{*2}-2iY_1QT^{*}+2iQY_1QT^{*}+2iQT^{*}Y_1Q-4QY_1QY_1Q+\\
&&4Y_1QY_1Q,\\
TT^{*}Q=(QTT^{*})^{*}&=&QT^{*2}+2iQY_1QT^{*}-2iTY_1Q+2iQT^{*}Y_1Q-4QY_1QY_1Q.
\end{eqnarray*}

\medskip

\noindent Accordingly,
\begin{eqnarray*}
iQ[T^{*},T]+i[T^{*},T]Q&=&iQT^{*}T-iQTT^{*}+iT^{*}TQ-iTT^{*}Q\\
&=&-2TQY_1+4iQY_1QY_1+2QY_1T^{*}+2Y_1QT^{*}+4iY_1QY_1Q-\\
&&2TY_1Q\\
&=&-2T(QY_1+Y_1Q-2QY_1Q)+2(QY_1+Y_1Q-2QY_1Q)T^{*}-\\
&&4X_1QY_1Q+4QY_1QX_1\\
&=&-2T(QY_1+Y_1Q-2QY_1Q)+2(QY_1+Y_1Q-2QY_1Q)T^{*}+\\
&&4[QY_1Q,X_1].
\end{eqnarray*}

\medskip

\noindent Now, a little computation yields that
\begin{align*}
iQ[T^{*},T]+i[T^{*},T]Q-2iQ[T^{*},T]Q=-&2T(QY_1+Y_1Q-2QY_1Q)+2(QY_1+\\
&Y_1Q-2QY_1Q)T^{*}+4[QY_1Q,X_1]+\\
&4Q[X_1,Y_1]Q.
\end{align*}

\medskip

\noindent Since $Q\in\mathcal{S}_T$, it follows that
\begin{align*}
iQ[T^{*},T]+i[T^{*},T]Q-2iQ[T^{*},T]Q=&2iT[Q,X_1]-2i[Q,X_1]T^{*}+\\
&4[QY_1Q,X_1]+4Q[X_1,Y_1]Q.
\end{align*}

\noindent The assumption $QX_1=X_1Q$ yields that $[QY_1Q,X_1]=-Q[X_1,Y_1]Q$ and therefore,
\begin{align*}
iQ[T^{*},T]+i[T^{*},T]Q-2iQ[T^{*},T]Q=0.
\end{align*}
Consequently, $Q\in\mathcal{S}_{X_1+i[T^{*},T]}$ as we wished to show.
\end{proof}	
	
\medskip
	
\begin{theorem}\label{theo2}
Let $T\in\mathcal{L}(\mathcal{H})$ and let $Q\in\mathcal{S}_T$ be a self-adjoint operator different from $0$ and $I$ with $QX_1=X_1Q$. Then $QY_1^2=Y_1^2Q$ and the operator $\tilde{T}=X_1+iY_1^2$ has nontrivial closed reducing subspaces.
\end{theorem}

\begin{proof}
First, note that if $Q$ is a multiple of the identity, then $T$ is self-adjoint (i.e., $Y_1=0$) and the result trivially follows. Accordingly, we may assume $Q$ is not a multiple of the identity operator.

\smallskip
	
Since $Q\in\mathcal{S}_T$ and $QX_1=X_1Q$, it holds $Q\in\mathcal{S}_{iY_1}$. In \cite[Corollary 6.3]{clemgallard} it was proved that any self-adjoint operator $Q\in \mathcal{S}_W$ for some $W\in\EL(\HH)$ verifies that $Q\in\mathcal{S}_{W^n}$ for every $n\in\N$. Then $Q\in S_{(iY_1)^2}=S_{-Y_1^2}$  and this implies $QY_1^2=Y_1^2Q$.
Hence $Q\tilde{T}=\tilde{T}Q$ and the result follows.
\end{proof}
	
For operators with a non-negative imaginary part we deduce the following:

\begin{corollary}\label{correducingpositiveimaginarypart}
	Let $T=X_1+iY_1$ with $Y_1\geq 0$. If $Q\in\mathcal{S}_T$ is self-adjoint and $QX_1=X_1Q$, then $Y_1Q(I-Q)=0$.
\end{corollary}

\begin{proof}
Note that Theorem \ref{theo2} implies $QY_1^2=Y_1^2Q$. Since $Y_1\geq 0$, then $QY_1=Y_1Q$. Hence $Q\in\mathcal{S}_T$ implies $Y_1Q(I-Q)=0$, as claimed.
\end{proof}

\begin{corollary}\label{correducinghyponormal}
Let $T\in\mathcal{L}(\mathcal{H})$ be hyponormal and let $Q\in\mathcal{S}_T$ be a self-adjoint operator different from $0$ and $I$ such that $QX_1=X_1Q$. Then either $T$ is normal or $Q(I-Q)$ is not a quasiaffinity. In particular, $T$ has  nontrivial closed reducing subspaces. Moreover, $Q$ is a projection provided $[T^{*},T]$ is injective.
\end{corollary}

\begin{proof}
By Lemma \ref{lemma1} and Corollary \ref{correducingpositiveimaginarypart}  we have
$$[T^{*},T]Q(I-Q)=0.$$
Consequently, if $T$ is not normal, the kernel $\ker Q(I-Q)$ is a nontrivial closed reducing subspace for $T$. Indeed, since $Q$ commutes with $X_1$, $\ker Q(I-Q)$ is invariant under $X_1$. Since $\ker Q(I-Q)=\ker Q\oplus \ker (I-Q)$ and both $\ker Q$ is invariant under $T^{*}$ and $\ker (I-Q)$ is invariant under $T$ yields the result and the statement of the Corollary \ref{correducinghyponormal} follows.
\end{proof}

Likewise, it is possible to deduce the existence of reducing subspaces for those operators with a hyponormal even power: 	
	
\begin{corollary}\label{corollary fin}
Let $T\in\EL(\HH)\setminus\C I$ and assume there exists some $m\in\N$ such that $T^{2^m}$ is hyponormal. If there exists some $Q\in\mathcal{S}_T$ self-adjoint with $QX_1=X_1Q$, then either $T^{2^m}$ is normal or $Q(I-Q)$ is not a quasiaffinity. In particular, $T$ has  nontrivial closed reducing subspaces.  Moreover, $Q$ is a projection provided $[T^{*2^m},T^{2^m}]$ is injective.
\end{corollary}

\begin{proof}
By means of the fact that $X_2=X_1^2-Y_1^2$, Theorem \ref{theo2} yields that $QX_2=X_2Q$ whenever $Q\in\mathcal{S}_T$ is self-adjoint and $QX_1=X_1Q$. Now, \cite[Corollary 6.3]{clemgallard} ensures that $Q\in \mathcal{S}_{T^{2^m}}$. Hence, $QX_{2^m}=X_{2^m}Q$ and upon applying Corollary \ref{correducinghyponormal}, the result follows.
\end{proof}
	
With the previous results in hand, we may state the characterization of the existence of reducing subspaces for hyponormal operators by means of the class $\mathcal{S}_T$ as follows:

\begin{theorem}
Let $T\in\EL(\HH)$ be a hyponormal operator with $T=X_1+iY_1$. Then $T$ has a nontrivial reducing subspace if and only if there exists a self-adjoint operator $Q\in\mathcal{S}_T$ different from  $0$ and $I$ such that $QX_1=X_1Q$.
\end{theorem}

\section*{Acknowledgements}

This work was initiated during a research stay of the authors in the Department of Mathematics at the University of Iowa. The authors would like to express their sincere gratitude to Prof. Ra\'ul Curto and the department for its warm hospitality and stimulating research environment, which greatly contributed to the development of this work. Both authors are partially supported by Plan Nacional I+D Grant No. PID2022-137294NB-I00, Spain, the Spanish Ministry of Science and Innovation, through the ''Severo Ochoa Proggrame for Centres of Excellence in R\& D'' (CEX2019-000904-S), and from the Spanish National Research Council, through the ''Ayuda extraordinaria a Centros de Excelencia Severo Ochoa'' (20205CEX001). The second author also acknowledges support of the Grant PRE2022-101849 funded by: MICIU/AEI/10.13039/501100011033 and FSE+.


\end{document}